\documentclass[a4paper, 11pt, headings = small, abstract]{scrartcl}

\usepackage[english]{babel}
\usepackage[utf8]{inputenc}
\usepackage{geometry}

\usepackage[semibold]{libertine}
\usepackage[upint,amsthm]{libertinust1math}

\usepackage{amsmath}
\usepackage{amsfonts}
\usepackage{amssymb}
\usepackage{amsthm}
\usepackage{mathtools}
\usepackage{paralist}
\usepackage{comment}

\usepackage[authoryear]{natbib}

\usepackage[dvipsnames]{xcolor}
\usepackage[colorlinks=true,linkcolor=blue,citecolor=blue,pdfborder={0 0 0}]{hyperref}
\usepackage{orcidlink}

\usepackage{dsfont}
\usepackage{bm}

\usepackage{graphicx}
\usepackage{enumerate}

\usepackage{tikz}
\usetikzlibrary{positioning}
\usetikzlibrary{backgrounds}

\newtheorem{theorem}{Theorem}[section]

\newtheorem{lemma}[theorem]{Lemma}

\theoremstyle{definition}

\theoremstyle{remark}

\allowdisplaybreaks[4]

\numberwithin{equation}{section}

\newcommand{\R}{\mathbb{R}}

\newcommand{\N}{\mathbb{N}}

\newcommand{\Prob}{\mathbb{P}}    

\newcommand{\Var}{\operatorname{Var}}
\newcommand{\Cov}{\operatorname{Cov}}

\newcommand{\diff}{\,\mathrm{d}}

\newcommand{\dbl}{{\ensuremath {\operatorname{db} }}}
\newcommand{\sbl}{{\ensuremath {\operatorname{sb} }}}

\begin{document}

\title{\fontsize{16}{19} On the maximal correlation coefficient for the bivariate Marshall Olkin distribution}

\author{
Axel B\"ucher\thanks{Ruhr-Universität Bochum, Fakultät für Mathematik. Email: \href{mailto:axel.buecher@rub.de}{axel.buecher@rub.de}} \orcidlink{0000-0002-1947-1617}
\and
Torben Staud\thanks{Ruhr-Universität Bochum, Fakultät für Mathematik. Email: \href{mailto:torben.staud@rub.de}{torben.staud@rub.de}}
}

\date{\today}

\maketitle

\begin{abstract} 
We prove a formula for the maximal correlation coefficient of the bivariate Marshall Olkin distribution that was conjectured in Lin, Lai, and Govindaraju (2016, Stat. Methodol., 29:1–9). The formula is applied to obtain a new proof for a variance inequality in extreme value statistics that links the disjoint and the sliding block maxima method.
\end{abstract}

\color{black}

\section{Introduction}
\label{sec:intro}

The bivariate Marshall Olkin exponential distribution \citep{MarOlk67} arises from considering random lifetimes within a two-component system, say $(X_1, X_2)$, where the components are subject to three different sources of fatal shocks. The occurrence times of the shocks are modelled by three independent exponential variables $Z_1, Z_2, Z_{12}$  with positive parameters $\lambda_1, \lambda_2, \lambda_{12}$, respectively. The first component of the system fails as soon as any of the two shocks $Z_1$ or $Z_{12}$ has occurred, that is, at time $X_1 = Z_1 \wedge Z_{12}$. Likewise, the second component fails at time $X_2 = Z_{2} \wedge Z_{12}$. A straightforward calculation then shows that the joint survival function of $(X_1, X_2)$ is
\begin{align*}
\bar H(x_1, x_2) 
= \Prob(X_1 > x_1, X_2 > x_2) &= \Prob(Z_1 > x_1, Z_2 > x_2, Z_{12} > x_1 \vee  x_2)  \\
&= \exp\{  - \lambda_1 x_1  - \lambda_2 x_2  - \lambda_{12}(x_1 \vee x_2)  \} \qquad (x_1, x_2 > 0),
\end{align*}
while the marginal survival functions satisfy $\bar H_j(x_j) = \Prob(X_j>x_j) = \exp\{-(\lambda_j + \lambda_{12}) x_j \}$.
In particular, the marginals are exponentially distributed. 

The Marshall Olkin distribution has been well-studied in the literature, with precise formulas being available for its Laplace transform, its product moments, or its Pearson, Kendall or Spearman correlation. We refer to \cite{MarOlk67, Lin2016}, among others.
The present work is motivated by an open conjecture mentioned in \cite{Lin2016} which concerns the maximal correlation coefficient of the Marshall Olkin distribution (see their open problem B). Recall that the maximal correlation coefficient is defined as
\begin{align}
\label{eq:maxcorr}
R(H) := R(X_1,X_2) := \sup_{f,g} \mathrm{Corr}(f(X_1), g(X_2)),
\end{align}
where the supremum is taken over all functions $f$ and $g$ such that $\Var(f(X_1)), \Var(g(X_2)) \in (0,\infty)$ exists and where $H$ denotes the cumulative distribution function (cdf) of $(X_1, X_2)$. In the general, the calculation of maximal correlation coefficients is difficult, but based on extensive moment calculations, \cite{Lin2016} conjecture that 
\begin{align} \label{eq:mo-corr}
R(X_1, X_2) = \frac{\lambda_{12}}{\sqrt {\lambda_1 + \lambda_{12}} \sqrt {\lambda_2 + \lambda_{12}}  }
\end{align}
for the bivariate Marshall Olkin distribution.
The main result of this note is a proof, given in Section~\ref{sec:mo-corr}. Our proof is based on certain elegant arguments from \cite{Yu2008}, who derived a new proof of the Gebelein-Lancaster theorem. The latter states that the maximum correlation of the
bivariate Gaussian distribution with correlation 
$\rho$ is equal to $|\rho|$.

Next to the proof of \eqref{eq:mo-corr}, a major contribution of this note is an application of \eqref{eq:mo-corr} to provide a new and elegant proof for an important variance inequality in extreme value statistics. Details are provided in Section~\ref{sec:extremes}.

\section{The maximal correlation for the Marshall Olkin distribution}
\label{sec:mo-corr}

In view of the continuity of the marginal survival functions, Sklar's theorem implies that the random vector $(X_1, X_2)$ has a unique survival copula $\hat C$, that is, a bivariate cdf $\hat C$ with standard uniform margins, such that $\bar H(x_1,x_2) = \hat C (\bar H_1(x_1), \bar H_2(x_2))$ for all $x_1, x_2 \ge 0$ \citep{Nel06}. A straightforward calculation shows that this copula is given by $\hat C = C_{\phi, \psi}$, where
\begin{align} \label{eq:mo-cop}
C_{\phi, \psi}(u,v) = \min(u^{1-\phi} v, uv^{1-\psi}), \quad u,v \in [0,1]^2,
\end{align}
with $\phi = \lambda_{12} / (\lambda_{1}+ \lambda_{12})$ and $\psi = \lambda_{12} / (\lambda_{2}+ \lambda_{12})$, see also \cite{Lin2016, Emb01}. Since the maximal correlation coefficient is invariant under (square-integrable) transformations of the margins, we have $R(H)=R(\hat C)=R(C_{\phi,\psi})$ and hence Equation~\eqref{eq:mo-corr} is an immediate corollary of the following theorem.

\begin{theorem} \label{theo:main}
For parameters $\phi, \psi \in [0,1]^2$, we have $R(C_{\phi, \psi}) = \sqrt{\phi \psi}$.
\end{theorem}

The proof is based on the following two lemmas from \cite{Yu2008}, which we quote in full for the sake of readability.

\begin{lemma}[\citealp{Yu2008}]\label{lem:yu1}
If non-degenerate random variables $X$ and $Y$ are conditionally independent given $Z$, then
$
R(X,Y) \le R(X,Z)R(Y,Z).
$
Moreover, equality holds if $( X , Z )$ and $( Y , Z )$ have the same distribution.
\end{lemma}

\begin{lemma}[\citealp{Yu2008}]\label{lem:yu2}
If non-degenerate random variables $X$ and $Y$ are independent and identically distributed, and $Z = f(X,Y)$, where $f$ is a symmetric function of $x$ and $y$, then $R(X,Z) \le 2^{-1/2}$.
\end{lemma}

\begin{proof}[Proof of Theorem~\ref{theo:main}]
The assertion is trivial if either $\phi\in\{0,1\}$ or $\psi \in \{0,1\}$; so let $\phi, \psi \in (0,1)$. 
For $k\ge 0$, define $f_k(x) = x^{k+1} / (k+1)$ with derivative $f_k'(x) = x^k$. A straightforward calculation yields $\Var(f_k(U))=\{(2k+3)(k+2)^2\}^{-1}$. Further, by an extension of Hoeffding's covariance formula \citep[Theorem 3.1]{Lo2017}, we have
\begin{align*}
\Cov(f_k(U), f_\ell(V) ) 
&= 
\int_0^1 \int_0^1 \big\{ \Prob(U >u, V >v) - (1-u)(1-v)  \big\} u^k v^\ell \diff u \diff v \\
&=
\int_0^1 \int_0^1 \big\{ \min(u^{1-\phi}v, uv^{1-\psi}) - uv  \big\} u^k v^\ell \diff u \diff v \\
&=
\frac{\phi \psi}{(k+2)(\ell+2) \{ (k+2)\phi + (\ell+2) \psi- \phi \psi \}},
\end{align*}
where the last equation follows from elementary calculations. As a consequence
\begin{align} \label{eq:int3}
\mathrm{Corr}(f_k(U), f_\ell(V)) = 
\frac{\phi \psi \sqrt{(2k+3)(2\ell+3)} }{  (k+2)\phi + (\ell+2) \psi- \phi \psi }.
\end{align}
Letting $k=\psi m$ and $\ell = \phi m$, this expression converges to $\sqrt{\phi \psi}$ for $m \to \infty$. As a consequence, $R(C_{\phi, \psi}) = R(U,V) \ge\sqrt{\phi \psi}$.

For the reverse inequality, let $X,Y,Z$ be independent standard uniform on $[0,1]$, and define
\[
(U,V) = (X^{1/(1-\phi)} \vee Z^{1/\phi}, Y^{1/(1-\psi)} \vee Z^{1/\psi}),
\]
which has cdf $C_{\phi, \psi}$. 
Further note that $U$ is conditionally independent of $V$ given $Z$. As a consequence, by Lemma~\ref{lem:yu1},
\[
R(U, V)  \le R(U, Z) R(V, Z).
\]
Note that $(U,Z)$ has joint cdf $D_\phi(u,v) = u^{1-\phi} (u^\phi \wedge v)$, and that, likewise, $(V,Z)$ has joint cdf $D_{\psi}$. It hence remains to show that $r(\xi) := R(D_\xi) \le \sqrt{\xi}$ for all $\xi \in (0,1)$.
In fact, we will show $r(\xi) = \sqrt{\xi}$, since we need `$\ge$' in the proof of `$\le$'.

We start by proving $r(\xi) \ge \sqrt{\xi}$. For that purpose, reconsider the function $f_k(x) = x^{k+1}/(k+1)$ with $k \ge 0$. A similar elementary calculation as for the proof of \eqref{eq:int3} shows that
\begin{align} \label{eq:int2}
\mathrm{Corr}(f_k(S), f_{k\xi}(T)) = \xi \frac{\sqrt{(2k+3)(2k \xi+3)}}{2k \xi+ \xi+2}, \quad (S,T) \sim D_\xi,
\end{align}
which converges to $\sqrt \xi$ for $k\to \infty$ and hence implies $r(\xi) \ge \sqrt{\xi}$.

For the proof of $r(\xi) \le \sqrt{\xi}$, recall $X,Y,Z$ from the beginning of the proof, and for $\xi_1, \xi_2 \in (0,1)$, let
$ \tilde W = Z$, $\tilde V =  Y^{1/(1-\xi_2)} \vee Z^{1/\xi_2}$ and $\tilde U = X^{1/(1-\xi_1)} \vee \tilde V^{1/\xi_1}$, i.e., 
\begin{align*}
(\tilde U, \tilde V, \tilde W)   
&=
(X^{1/(1-\xi_1)} \vee Y^{1/\{\xi_1(1-\xi_2)\}} \vee Z^{1/(\xi_1 \xi_2)}, Y^{1/(1-\xi_2)} \vee Z^{1/\xi_2}, Z).
\end{align*}
A straightforward calculation shows that $(\tilde U, \tilde V)$ has cdf $D_{\xi_1}$, that $(\tilde V, \tilde W)$ has cdf $D_{\xi_2}$ and that $(\tilde U, \tilde W)$ has cdf $D_{\xi_1 \xi_2}$.
Furthermore, $\tilde U$ and $ \tilde W$ are conditionally independent given $ \tilde V$, whence, by Lemma~\ref{lem:yu1},
\begin{align} \label{eq:ri}
r(\xi_1 \xi_2) = R(\tilde U, \tilde W) \le R(\tilde U,\tilde V) R(\tilde W,\tilde V) = r(\xi_1) r(\xi_2).
\end{align}
This implies monotonicity of $\xi \mapsto r(\xi)$. Furthermore, by setting $\xi_1=\xi_2=\xi$, we get equality in the previous display by Lemma~\ref{lem:yu1}, which yields
\begin{align} \label{eq:funr}
r(\xi)=r(\xi^2)^{1/2}.
\end{align}
Next, an application of Lemma~\ref{lem:yu2} gives $r(1/2) \le 2^{-1/2}$,
which, in view of the previous display, implies that $r(2^{-1/2}) = r(1/2)^{1/2} \le 2^{-1/4}$. Since we have already shown $r(\xi) \ge \sqrt \xi$ for all $\xi$, we obtain that 
$r(2^{-1/2})=2^{-1/4}$. 

For $m\in\N$, we may apply \eqref{eq:ri} $m$-times to obtain that
\[
2^{-m/4} \le r(2^{-m/2}) \le r(2^{-1/2})^{m} = 2^{-m/4},
\]
whence $r(2^{-m/2}) = 2^{-m/4}$. Next, for any $n\in\N$, we may apply \eqref{eq:funr} $(n-1)$-times to obtain that
\[
r(2^{-m/2^n}) = r(2^{-m/2^{n-1}})^{1/2} = r(2^{-m/2^{n-2}})^{1/2^2} = \dots = r(2^{-m/2})^{1/2^{n-1}} = (2^{-m/4})^{1/2^{n-1}}  = (2^{-m/2^n})^{1/2}.
\]
We have hence shown that $r(x)=\sqrt x$ for all $x \in \mathcal C := \{2^{-m/2^n} \in (0,1): n,m\in\N\}$. For any fixed $\xi\in(0,1)$, we can choose sequences $(x_k)_k, (y_k)_k$ in $\mathcal C$ converging to $\xi$ such that $x_k \le \xi \le y_k$ for all $k$. Hence, by monotonicity of $r$, $\sqrt{x_k}=r(x_k) \le r(\xi) \le r(y_k) = \sqrt {y_k}$, which implies $r(\xi)=\xi$ by taking the limit for $k\to\infty$. This finalizes the proof.
\end{proof}

\section{An application in extreme value statistics}
\label{sec:extremes}

The Marshall Olkin copula from \eqref{eq:mo-cop} is easily seen to be max-stable, that is, we have 
\[
C_{\phi, \psi}(u,v)=C_{\phi, \psi}(u^{1/m}, v^{1/m})^m \qquad \forall u,v\in[0,1], m \in \N.
\]
As a consequence, it is an extreme-value copula \citep{GudSeg10} and may hence appear as the weak limit copula of affinely standardized bivariate maxima. In fact, it happens to occur as a limit in the following simple situation: let $(X_n)_n$ denote an independent and identically distributed (iid) sequence of random variables satisfying the standard domain of attraction (DOA) condition \citep{DehFer06} that 
$(\max_{i=1}^r X_i - b_r) / a_r$ converges weakly to a non-degenerate limit distribution for $r\to\infty$, where $(b_r)_r \subset \R$ and $(a_r)_{r} \subset (0,\infty)$ are suitable scaling sequences. In that case, by the Fisher-Tippett-Gnedenko Theorem \citep{FisTip28, Gne43}, the limit distribution is necessarily the generalized extreme value distribution with cdf $G_\gamma(x) = \exp\{  (1+\gamma x)^{-1/\gamma}\} $ for $x$ such that $1+\gamma x > 0$; here, $\gamma\in\R$ denotes the extreme value index. Now, under the DOA condition, we have, for any $\zeta \in [0,1]$ and writing $\zeta_r = \lfloor r \zeta \lfloor$, 
\begin{align} \label{eq:max2}
\lim_{r \to \infty} \Prob\Big( \frac{\max_{i=1}^r X_i - b_r}{a_r} \le x, \frac{\max_{i= \zeta_r +1 }^{ \zeta_r + r} X_i - b_r}{a_r}  \le y\Big) = G_{\zeta, \gamma}(x,y) :=  C_{1-\zeta, 1-\zeta}\{ G_\gamma(x), G_\gamma(y) \}  
\end{align}
for all $x,y \in \R$; see, for instance, Lemma B.3 in \cite{BucZan23}. In fact, the result in \eqref{eq:max2} even holds if the iid sequence is replaced by a stationary time series, provided the long range dependence is suitably controlled.

The weak convergence in \eqref{eq:max2} is fundamental for the so-called sliding block maxima method in extreme value statistics. We refer to \cite{BucSeg18-sl, ZouVolBuc21, BucZan23} among others for details on the general approach. As it happens, even for time series data, the asymptotic behavior of respective estimators is typically driven by certain empirical means satisfying a central limit theorem with asymptotic variance formula given by
\[
\sigma^2_\sbl(h) := 2 \int_0^1 \Cov(h(Y_{1, \zeta}), h(Y_{2, \zeta})) \diff \zeta,
\]
where $(Y_{1, \zeta}, Y_{2, \zeta}) := (Y_1, Y_2) \sim G_{\zeta, \gamma}$ with $G_{\zeta, \gamma}$ from \eqref{eq:max2} and where $h$ is square-integrable with respect to $G_\gamma$. 
On the other hand, the traditional (disjoint) block maxima method satisfies respective limit theorems with asymptotic variance given by 
\[
\sigma^2_\dbl(h) := \Var(h(Y_1)), 
\]
where $Y_1 \sim G_\gamma$. The following theorem is essential for showing that the sliding block maxima method is statistically more efficient than the traditional disjoint block maxima method (again, we refer to the aforementioned references). The first part can be deduced from a technical result in  \cite{ZouVolBuc21}, see their Lemma A.10, but Theorem~\ref{theo:main} above offers the possibility for an elegant and short proof.

\begin{theorem}
For $h\colon \R \to \R$ with $ \int h(x)^2 \diff G_\gamma(x) < \infty$, we have $\sigma^2_\sbl(h) \leq \sigma^2_\dbl (h)$. Moreover, equality holds if and only if $h$ is a function such that $\mathrm{Corr}(h(Y_{1, \zeta}), h(Y_{2, \zeta})) = R(G_{\gamma, \zeta}) = 1-\zeta$ for Lebesgue almost every value of $\zeta \in [0,1]$.
\end{theorem}

\begin{proof}  We have 
\begin{align*}
\sigma^2_\sbl(h) = 2 \int_0^1 \Cov(h(Y_{1,\zeta}), h(Y_{2, \zeta})) \diff \zeta
&= 
2 \sigma^2_\dbl (h) \int_0^1 \mathrm{Corr}(h(Y_{1,\zeta}), h(Y_{2, \zeta}))  \diff \zeta  \\
&\le 
2 \sigma^2_\dbl (h) \int_0^1 R(C_{1-\zeta, 1-\zeta}) \diff \zeta   \\
&= 
2 \sigma^2_\dbl (h) \int_0^1 1-\zeta \diff \zeta = \sigma^2_\dbl (h),
\end{align*}
where we used Theorem~\ref{theo:main} at the penultimate equality. The second statement is immediate. 
\end{proof}

\section*{Acknowledgements} 
Financial support by the German Research Foundation (DFG grant number 465665892) is gratefully acknowledged.
The authors are grateful to the participants of the Oberwolfach Workshop on ``Mathematics, Statistics, and Geometry of Extreme Events in High Dimensions'' for their valuable comments.

\bibliographystyle{apalike}
\bibliography{biblio}

\end{document}